\documentclass[11pt,oneside]{amsart}

\usepackage{amssymb}
\usepackage{amsmath}
\usepackage{amsthm}
\usepackage[all]{xy}

\theoremstyle{plain}
\newtheorem{thm}{Theorem}[section]
\newtheorem{prop}[thm]{Proposition}
\newtheorem{lem}[thm]{Lemma}
\newtheorem{cor}[thm]{Corollary}

\theoremstyle{definition}
\newtheorem{defn}[thm]{Definition}
\newtheorem{eg}[thm]{Example}

\theoremstyle{remark}
\newtheorem{rem}[thm]{Remark}

\DeclareMathOperator{\Pic}{Pic}
\DeclareMathOperator{\Cl}{Cl}

\DeclareMathOperator{\Nef}{Nef}
\DeclareMathOperator{\Eff}{Eff}
\DeclareMathOperator{\Mov}{Mov}

\DeclareMathOperator{\codim}{codim}

\DeclareMathOperator{\Bs}{Bs}

\DeclareMathOperator{\Cone}{Cone}

\DeclareMathOperator{\Sym}{Sym}

\def\Z{\mathbb{Z}}
\def\Q{\mathbb{Q}}
\def\R{\mathbb{R}}
\def\C{\mathbb{C}}

\def\r+{\mathbb{R}_{\geq 0}}

\def\r+{{\R}_{\geq 0}}
\def\q+{{\Q}_{\geq 0}}
\def\P{\mathbb{P}}
\def\arw{\rightarrow}
\def\*c{\C^{\times}}

\newcommand{\calo}{\mathcal {O}}

\makeatletter
  
  \@addtoreset{equation}{section}
\makeatother

\begin{document}

\title{Examples of Mori dream spaces with Picard number two}
\author{Atsushi Ito} 
\address{Graduate School of Mathematical Sciences, 
The University of Tokyo, 3-8-1 Komaba, 
Meguro, Tokyo, 153-8914, Japan.}
\email{itoatsu@ms.u-tokyo.ac.jp}

\begin{abstract}
In this note,
we give a sufficient condition such that a projective variety with Picard number two is a Mori dream space.
Using this condition,
we obtain examples of Mori dream spaces with Picard number two.
\end{abstract}


\maketitle

\section{Introduction}\label{section_intro}

Mori dream spaces, which were introduced by Hu and Keel in \cite{HK},
are special varieties which have very nice properties in view of the minimal model program.
In the paper,
Hu and Keel investigated properties of Mori dream spaces,
especially those related to GIT.

\vspace{2mm}
We recall the definition of Mori dream spaces.
For a normal $\Q$-factorial projective variety $X$,
we set 
\[
N^1(X)_K :=(\Cl (X) / \equiv ) \otimes_{\Z} K =  (\Pic(X) / \equiv ) \otimes_{\Z} K
\]
for $K=\Q$ or $\R$.
We denote by $\Eff(X)$, $\Mov(X)$, and $\Nef(X)$
the cones in $N^1(X)_{\R}$ generated by effective, movable, and nef divisors respectively.

\begin{defn}\label{def_SQM}
By a \textit{small $\Q$-factorial modification (SQM)}
of a projective variety $X$,
we mean a birational map $f : X \dashrightarrow X'$
with $X'$ projective, normal, and $\Q$-factorial,
such that $f$ is an isomorphism in codimension one.
\end{defn}

\begin{defn}\label{def_MDS}
A normal projective variety $X$ is called
a \textit{Mori Dream Space} if the following hold:
\begin{itemize}
\item[i)] $X$ is $\Q$-factorial and $\Pic(X)_{\Q} = N^1(X)_{\Q}$,
\item[ii)] $\Nef(X)$ is the affine hull of finitely many semiample line bundles,
\item[iii)] There is a finite collection of SQMs $f_i : X \dashrightarrow X_i$ such that each $X_i$ satisfies ii)
and $\Mov(X) = \bigcup_i f_i^{*}(\Nef(X_i))$.
\end{itemize}
\end{defn}

Quasi-smooth projective toric varieties are typical examples of Mori dream spaces.
In \cite[Corollary 1.3.2]{BCHM},
it is shown that $\Q$-factorial log Fano varieties are Mori dream spaces. 
See \cite{AHL}, \cite{Jo}, \cite{TVV}, etc. for other examples.

\vspace{1mm}
By definition,
a normal $\Q$-factorial projective variety $X$ is a Mori dream space if $\Pic(X)_{\Q} \cong \Q$.
Then,
how about the case when the Picard number is two?
The following theorem gives a sufficient condition for $X$ to be a Mori dream space when the Picard number is two.

\begin{thm}\label{intro thm}
Let $X$ be a normal $\Q$-factorial projective variety with Picard number $2$ and $\Pic(X)_{\Q} = N^1(X)_{\Q}$.
Assume that
there exist nonzero effective Weil divisors $D_1,\ldots, D_r$ and $D'_1,\ldots, D'_{r'}$ on $X$ 
for some $2 \leq r,r' \leq \dim X$ such that
\begin{itemize}
\item[a)] $\Cone (D_1,\ldots,D_{r}) \cap \Cone(D'_1,\ldots, D'_{r'}) = \{ 0 \} $,
\item[b)] $D_1 \cap \cdots \cap D_{r}= D'_1 \cap \cdots \cap D'_{r'}= \emptyset$.
\end{itemize}
Then $X$ is a Mori dream space.
\end{thm}

\vspace{3mm}
As corollaries of Theorem \ref{intro thm},
we obtain the following examples,
which are special cases of Corollary \ref{cor_c.i.}.
In both cases,
we will construct divisors $D_i,D'_j$ as in Theorem \ref{intro thm}
by using the defining equations of $X$ or $Z$.
We note that Corollary \ref{cor 1} is proved by Oguiso (at least in the case $n=3$) in his private note \cite{Og}.

\begin{cor}[\cite{Og}]\label{cor 1}
Let $X \in |\calo_{\P^1 \times \P^n }(a,b)|$ be a general hypersurface on $\P^1 \times \P^n $
such that $a,b >0$ and $n \geq \max\{ a, 3\}$.
Then
$X$ is a Mori dream space.
\end{cor}

\begin{cor}\label{blowup of Fano}
Let $Z \subset \P^N$ be a complete intersection of general hypersurfaces of degrees $d_1, \ldots,d_s$,
and let $X \arw Z$ be the blow up at a general point $p \in Z$.
If $Z$ is a Fano variety,
i.e.,
if the anticanonical bundle $-K_Z = \calo_Z( N +1 - \sum_{i=1}^s d_i)$ is ample,
then $X$ is a Mori dream space.
\end{cor}

\begin{rem}
Let $X \in |\calo_{\P^k \times \P^{n+1-k} }(a,b)|$ be a normal $\Q$-factorial hypersurface on $\P^k \times \P^{n+1-k} $
such that $n \geq 3$ and $a,b >0$.
By the Lefschetz hyperplane theorem,
it holds that $\Pic(X) \cong \Pic( \P^k \times \P^{n+1-k}) $.

If $2 \leq k \leq n-1$,
$\calo_X(1,0) $ and $ \calo_X(0,1) $ are semiample and not big.
Hence $X$ is a Mori dream space since $\Eff(X)=\Mov(X)=\Nef(X)= \r+ \calo_X(1,0) + \r+  \calo_X(0,1) $.
We note that this is a special case of \cite[Corollary 2]{Jo}.
Since $K_X= \calo_X(a -k-1, b-(n+2-k))$,
it holds that
\begin{align*}
\kappa(X) = \left\{ 
\begin{array}{cl} 
 - \infty &  \text{for } a \leq k \text{ or } b \leq n+1 -k, \\
 0 & \text{for }  a= k+1, b= n+2 -k,\\
 k  & \text{for }  a \geq k+2, b= n+2 -k,\\
 n & \text{for }  a \geq k+2, b \geq  n+3 -k,\\
\end{array} \right.
\end{align*}
where $\kappa(X)$ is the Kodaira dimension of $X$.
Thus $\kappa(X)$ can be any number in $\{ - \infty, 0, 1, 2, \cdots,n\}$ except $1$
by choosing suitable $k, a,b$.

On the other hand,
a general hypersurface $X \subset \P^1 \times \P^{n}$ in $|\calo_{\P^1 \times \P^n}(3,n+1)|$ for $n \geq 3$
is a Mori dream space with $\kappa(X)=1$ by Corollary \ref{cor 1}.

Hence for any $n \geq 3$ and $\kappa \in \{ - \infty, 0, 1, \cdots,n\} $,
there exists an $n$-dimensional smooth Mori dream space with $\kappa(X)=\kappa$ and Picard number $2$.
This is shown in \cite{Og} when $n=3$.
\end{rem}

\begin{rem}
In a recent paper \cite{Ot},
Ottem independently studied hypersurfaces in $\P^1 \times \P^n$ more precisely by a different argument.
Ottem determines for which $a,b >0$,
a (very) general hypersurface $X \in |\calo_{\P^1 \times \P^n }(a,b)|$
is a Mori dream space,
and describes the Cox ring of such $X$ when $X$ is a Mori dream space.
\end{rem}

\subsection*{Notation}
Throughout this note,
we work over the complex number field $\C$.
A divisor means a Weil divisor.
For a base point free divisor $D$,
we denote by $\varphi_{|D|}$ the morphism defined by the complete linear system $|D|$.
We define $\dim \emptyset = -1$.

\subsection*{Acknowledgments}
The author would like to express his gratitude to Professor Yujiro Kawamata
for his valuable comments and advice.
He is also grateful to Professor Keiji Oguiso for giving him his private note.
He wishes to thank Professors Yoshinori Gongyo and Shinnosuke Okawa for
useful comments and answering many questions about Mori dream spaces.


\section{Proof of Theorem \ref{intro thm}}\label{section_proof}


In this section,
we prove Theorem \ref{intro thm}.
We will construct all the SQMs inductively by using the divisors $D_i,D'_j$.

\begin{defn}
Let $X$ be a normal $\Q$-factorial projective variety with Picard number $2$ and $\Pic(X)_{\Q} = N^1(X)_{\Q}$.
Fix a divisor or a line bundle $A \in N^1(X)_{\R} \setminus \{0\}$.
For nonzero effective divisors $D_1, D_2$ on $X$,
we denote
\[
D_1 \succeq_A D_2 \quad \text{if} \quad D_2 \in \Cone ( D_1, A) .
\]
\end{defn}

\begin{figure}[htbp]
 \begin{center}
\[
\begin{xy}
(0,0)="O",
(25,10)="A",(20,20)="B",
(10,25)="C",(40,50)="D",
(10,35)="E",
(5,0)="F",(40,0)="G",
(20,7)="I",(20,33)="J",
(20,0)="K",
(-20,0)="1",(30,0)="2",
(0,-10)="3",(0,35)="4",
(28,12)*{A},(22.5,22.5)*{D_2},
(12,28)*{D_1}

\ar "O";"A"
\ar "O";"B"
\ar "O";"C"
\end{xy}
\]
 \end{center}
 \caption{$D_1 \succeq_A D_2 $}
 \label{figure2}
\end{figure}
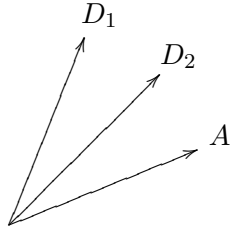

\begin{defn}\label{def_cond*}
Let $X$ be a normal $\Q$-factorial projective variety with Picard number $2$ and $\Pic(X)_{\Q} = N^1(X)_{\Q}$.
Let $D_1 , \ldots, D_{k+1} $ be nonzero effective divisors  on $X$ for $ 0 \leq k \leq \dim X -1$.
We say that $D_1 , \ldots, D_{k+1} $ satisfy condition $(*)$
if $D_1 \succeq_A \cdots \succeq_A D_{k+1} $ for an ample line bundle $A$
and the following conditions 1) - 4) hold.
\begin{itemize}
\item[1)] $D_{k+1}$ is semiample,
\item[2)] $D_{k+1} \not \in \Cone (D_1, \cdots , D_k) $,
\item[3)] $D_1 \cap \cdots \cap D_k \not = \emptyset$,
\item[4)] $\dim \varphi_{|m D_{k+1}|} (D_1 \cap \cdots \cap D_k) \leq \dim X -k -1$ for sufficiently divisible $m > 0$,
\end{itemize}
where we set $\Cone (D_1,\ldots,D_{k}) = \{0\}$ and $D_1 \cap \cdots \cap D_{k} =X$ when $k=0$.
\end{defn}

If $D_1 , \ldots, D_{k+1}$ satisfy condition $(*)$,
$D_{k+1} $ is nef but not ample as we will prove in the following lemma.
Hence $\r+ D_{k+1} $ is an edge of $\Nef(X)$.

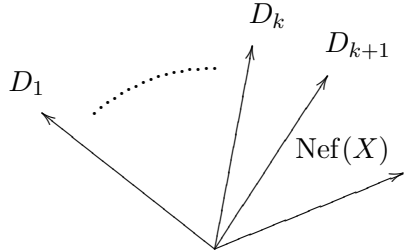
\begin{figure}[htbp]
 \begin{center}
\[
\begin{xy}
(0,0)="O",
(25,10)="A",(15,23)="B",
(5,27)="C",(-23,18)="D",
(10,35)="E",
(5,0)="F",(40,0)="G",
(20,7)="I",(20,33)="J",
(20,0)="K",
(-20,0)="1",(30,0)="2",
(0,-10)="3",(0,35)="4",
(19,27)*{D_{k+1}},(7,31)*{D_k},
(-25,22)*{D_1},
(17,13)*{\Nef(X)},
(-16,18);(0,24)

\ar "O";"A"
\ar "O";"B"
\ar "O";"C"
\ar "O";"D"

\crv{~*{.}(-9,24)}
\end{xy}
\]
 \end{center}
 \caption{Condition $(*)$}
 \label{figure3}
\end{figure}

\begin{lem}\label{rem_not_ample}
Let $X$ be a normal $\Q$-factorial projective variety with Picard number $2$ and $\Pic(X)_{\Q} = N^1(X)_{\Q}$.
Assume that divisors $D_1 , \ldots, D_{k+1} $ on $X$ satisfy condition $(*)$.
Then $D_{k+1}$ is not ample.

Furthermore,
if $k=0$, $D_{k+1}$ is not big.
If $k >0$, $D_{k+1}$ is big
and $\varphi_{|m D_{k+1}|} $ is an isomorphism on $X \setminus (D_1 \cap \cdots \cap D_k)$ for sufficiently divisible $m>0$.
\end{lem}

\begin{proof}
It holds that
\[
\dim \varphi_{|m D_{k+1}|} (D_1 \cap \cdots \cap D_k)  \leq \dim X- k -1 < \dim D_1 \cap \cdots \cap D_k,
\]
where the latter inequality follows from condition 3) and the Krull's principal ideal theorem
(note that $D_1, \ldots,D_k$ are $\Q$-Cartier).
Hence $D_{k+1}$ is not ample.

If $k=0$,
$D_{k+1}$ is not big by condition 4) and $D_1 \cap \cdots \cap D_k=X$.
If $k > 0$,
we have $D_{k+1} \sim_{\Q} \alpha_i D_i + \beta_i A$ for some $\alpha_i, \beta_i > 0$ and $1 \leq i \leq k$
by condition 2).
Hence $D_{k+1}$ is big and $\varphi_{|m D_{k+1}|}  $ is an isomorphism outside $D_i$ for any $1 \leq i \leq k$
since $m \beta_i A$ is very ample.
\end{proof}

\begin{lem}\label{exist_of edge}
Let $X$ be a normal $\Q$-factorial projective variety with Picard number $2$ and $\Pic(X)_{\Q} = N^1(X)_{\Q}$.
Let $A$ be an ample line bundle
and let $D_1 \succeq_A \cdots \succeq_A D_r $ be nonzero effective divisors  on $X$
such that $D_1 \cap \cdots \cap D_r =  \emptyset$ and $2 \leq r \leq \dim X$.
Then there exists $0 \leq k \leq r-1$ such that
$D_1,\ldots, D_{k+1}$ satisfy condition $(*)$.
\end{lem}

\begin{proof}
Since $ D_i \succeq_A D_r$ for $1 \leq i \leq r$,
there exist rational numbers $\alpha_i,\beta_i \geq 0$ such that
$D_r \sim_{\Q} \alpha_i D_i + \beta_i A$.
For sufficiently divisible $m$,
$\Bs(|m D_r|) \subset D_i$
since $m D_r \sim m \alpha_i D_i +  m \beta_i A$ and $m \beta_i A $ is base point free.
Hence we have $\Bs(|m D_r|) \subset D_1 \cap \cdots \cap D_r = \emptyset$,
which means $D_r$ is semiample.
Set $Z$ be the image of $X$ by $\varphi:=\varphi_{|m D_r|}$ for sufficiently divisible $m$.

Let $0 \leq k_1 \leq r-1$ be the smallest $i$ such that $\r+ D_{i+1} = \r+ D_r$.
Since $\Pic(X)_{\Q} = N^1(X)_{\Q}$,
$D_i$ is $\Q$-linearly equivalent to a positive multiple of $D_r$ for each $k_1 +1 \leq i \leq r$.
Hence $\varphi_* (D_i) $ is an ample divisor on $Z$
and $D_i = \varphi^* \varphi_* (D_i)$ holds for each $k_1 +1 \leq i \leq r$.

\vspace{2mm}
First, we assume $D_1 \cap \cdots \cap D_{k_1} \not = \emptyset$ and show that
$D_1, \cdots , D_{k_1+1}$ satisfy condition $(*)$.
Since $D_{k_1+1}$ is $\Q$-linearly equivalent to a positive multiple of $D_r$,
$D_{k_1+1}$ is semiample.
By the choice of $k_1$,
$D_{k_1+1} \not \in \Cone (D_1, \cdots , D_{k_1}) $ holds.
Hence it is enough to show the inequality
$\dim \varphi (D_1 \cap \cdots \cap D_{k_1})  \leq \dim X -k_1 -1$.

Since $D_1 \cap \ldots \cap D_r = \emptyset $ and $D_{k_1+1}, \ldots, D_r$ are pullbacks of divisors on $Z$,
we have
\[
\varphi (D_1 \cap \cdots \cap D_{k_1}) \cap \varphi_* D_{k_1+1} \cap \cdots \cap \varphi_* D_r = \emptyset.
\]
Hence it holds that
\[
\dim \varphi (D_1 \cap \cdots \cap D_{k_1}) \leq r-k_1-1   \leq \dim X - k_1 -1
\]
since $\varphi_* D_{k_1+1} , \ldots,  \varphi_* D_r$ are ample divisors on $Z$.
Hence $D_1 , \ldots, D_{k_1+1}$ satisfy condition $(*)$.

\vspace{2mm}
Next,
we assume $D_1 \cap \cdots \cap D_{k_1}  = \emptyset$.
In this case,
we replace $r$ with $k_1$ and repeat the above argument, and obtain $0 \leq k_2 \leq k_1-1 $.
Since $r > k_1 > k_2 > \cdots \geq 0$,
this process must stop and we obtain $k$ as in the statement of this lemma.
\end{proof}

The following is the key proposition to construct SQMs inductively in the proof of Theorem \ref{intro thm}.

\begin{prop}\label{exist_SQM}
Let $X$ be a normal $\Q$-factorial projective variety with Picard number $2$ and $\Pic(X)_{\Q} = N^1(X)_{\Q}$.
Let $D_1 , \ldots, D_{k+1} $ be effective divisors on $X$ which satisfy condition $(*)$.

If $\dim D_1 \cap \cdots \cap D_k \leq \dim X - 2$,
there exists an SQM $X  \dashrightarrow X^{+}$ 
and an integer $0 \leq l \leq k-1$ such that
\begin{itemize}
\item $D^+_{k+1} $ is semiample,
\item $\Nef(X^+) = \r+ D^+_{l+1} + \r+ D^+_{k+1} $,
\item $D^+_1 , \ldots,  D^+_{l+1} $ satisfy condition $(*)$ on $X^+$, 
\end{itemize}
where $D^+_{i}$ is the strict transform of $D_i$ on $X^+$ for each $1 \leq i \leq k+1$.
In particular,
$\Nef(X^+)$ is spanned by two semiample divisors $D^+_{l+1} , D^+_{k+1} $.
\end{prop}

\begin{proof}
In Step 1,
we construct $X^{\dagger}$,
which is isomorphic to $X$ in codimension $1$.
In Step 2,
we show that the normalization $X^{+}$ of $X^{\dagger}$ is an SQM of $X$.
In Step 3,
we show the existence of $l$ as in the statement.

\ \\
\textbf{Step 1.}
By 2) in condition $(*)$,
we can write
$D_{k} \sim_{\Q} a_i D_i + b_i D_{k+1}$ for some rational numbers $a_i > 0 , b_i \geq 0$ for each $1 \leq i \leq k$.
Fix a sufficiently divisible integer $m >0$
and take the scheme-theoretic intersection
\[
C : = \bigcap_{1 \leq i \leq k} m a_i D_i \subset X.
\]
Let $\phi : \widetilde{X} \arw X $ be the blow up along $C$
and let $E$ be the Cartier divisor on $\widetilde{X}$ such that $\phi^{-1} I_C = \calo_{\widetilde{X}} (-E)$
for the ideal sheaf $I_C$ of $C$.

Since $m D_k \sim m a_i D_i + m b_i D_{k+1}$ and $ m b_i D_{k+1}$ is base point free,
$\phi^* (m D_k) -E$ is base point free as well.
Consider the following commutative diagram.

\[
\xymatrix{
& \widetilde{X} \ar[dl]_{\phi} \ar[dr]^{\phi^{\dagger}} \ar@/^8mm/[ddrr]^f & & \\
 X \ar[dr]_{\pi} & & X^{\dagger} \ar[dl]^{\varpi^{\dagger}} \ar[dr]_{\pi^{\dagger}} &  \\
&  Z & & Z^{\dagger}   .\\
}\]
In the above diagram,
$\pi := \varphi_{|m D_{k+1}|}$,
$f :=\varphi_{|\phi^* (m D_k) -E|}$,
$Z=\pi(X)$,
$Z^{\dagger} = f(\widetilde{X})$,
$\phi^{\dagger} := (\pi \circ \phi) \times f : \widetilde{X} \arw Z \times Z^{\dagger}$,
$X^{\dagger} = \phi^{\dagger}(\widetilde{X})$,
and $ \varpi^{\dagger}$ and $\pi^{\dagger}$ are the first and second projections from $X^{\dagger} \subset  Z \times Z^{\dagger}$ respectively.

By Lemma \ref{rem_not_ample},
the restriction $ \pi |_{X \setminus C} : X \setminus C \arw Z \setminus \pi(C)$ is an isomorphism.
Since $\phi |_{ \widetilde{X} \setminus E} :  \widetilde{X} \setminus E 
\arw X \setminus C $ is also an isomorphism,
the restriction of $\varpi^{\dagger}$ on $X^{\dagger} \setminus \phi^{\dagger}(E)$ is an isomorphism onto $ Z \setminus \pi(C)$.
Hence $X \setminus C $ and $X^{\dagger} \setminus \phi^{\dagger}(E) $ are isomorphic.

By assumption,
we have $\codim_X C \geq 2 $ since $C=D_1 \cap \cdots \cap D_k$ as sets.
To show that $X $ and $X^{\dagger}$ are isomorphic in codimension 1,
it is enough to show that $\codim_{X^{\dagger}} \phi^{\dagger}(E) \geq 2$.
By definition,
there exists a natural surjection
\[
\bigoplus_{i=1}^k \calo_X(-m a_i D_i) |_C \arw I_{C} / I_{C}^2.
\]
Since $ m D_k -m a_i D_i \sim m b_i D_{k+1}$,
there exists an inclusion
\begin{align}
E \hookrightarrow \P_C \left( \bigoplus_{i=1}^k \calo_X(-m a_i D_i) |_C \right)
\cong  \P_C \left( \bigoplus_{i=1}^k \calo_X(m b_i D_{k+1}) |_C \right), \label{shiki_inclusion}
\end{align}
and we have
\begin{align}
\calo_{\widetilde{X}}( \phi^* mD_k - E) |_E \sim \calo_{\P_C \left( \bigoplus_{i=1}^k \calo_X(m b_i D_{k+1}) |_C \right)}(1) |_E \label{shiki_lin_eq}
\end{align}
under this inclusion.

From (\ref{shiki_inclusion}), (\ref{shiki_lin_eq}), and the definition of $\phi^{\dagger}$,
we have
\begin{align}
\phi^{\dagger} (E) \hookrightarrow \P_{\pi(C)} \left( \bigoplus_{i=1}^k \calo_Z (m b_i \overline{D}_{k+1}) |_{\pi(C)} \right) \label{shiki_inclusion2}
\end{align}
for a suitable scheme structure on $\pi(C)$,
where $\overline{D}_{k+1} $ is the divisor on $Z$ whose pullback on $X$ is $D_{k+1}$.
Hence 
\begin{align*}
\dim \phi^{\dagger}(E) &\leq \dim \P_{\pi(C)} \left( \bigoplus_{i=1}^k \calo_Z (m b_i \overline{D}_{k+1}) |_{\pi(C)} \right) \\
&=\dim \pi(C) + k-1 \leq \dim X^{\dagger}-2
\end{align*}
by 3) in condition $(*)$ since $\dim X^{\dagger} = \dim X$ and $C=D_1 \cap \cdots \cap D_k$ as sets.
Thus $ \codim_{X^{\dagger}} \phi^{\dagger}(E) \geq 2$ holds.
Hence $X$ and $X^{\dagger}$ are isomorphic in codimension $1$.

\ \\
\textbf{Step 2.}
Let $\nu : X^{+} \arw X^{\dagger}$ be the normalization.
Since $X^{\dagger} \setminus \phi^{\dagger}(E) \cong X \setminus C$ is normal,
$\nu$ is an isomorphism in codimension $1$.
Hence $ X  \dashrightarrow  X^+$ is also an isomorphism in codimension $1$.

To show that  $X^+$ is an SQM of $X$,
it suffices to see that $X^+$ is $\Q$-factorial.
Let $D^{\dagger}_i$ be the strict transform of $D_i$ on $X^{\dagger}$.
Since $D_{k+1} = \pi^* \overline{D}_{k+1}$
and $X$ and $X^{\dagger}$ are isomorphic in codimension $1$,
it holds that $D^{\dagger}_{k+1}= {\varpi^{\dagger}}^* \overline{D}_{k+1}$.
In particular,
$D^{\dagger}_{k+1} $ is $\Q$-cartier and semiample.
Since $ \varphi_{|\phi^* (m D_k) -E|} =f = \pi^{\dagger} \circ \phi^{\dagger}$,
there exists an ample divisor $H$ on $Z^{\dagger}$
such that $\phi^* (m D_k) -E=  f^*  H =  {\phi^{\dagger}}^{*}  {\pi^{\dagger}}^{*} H$.
Hence we have $  {\pi^{\dagger}}^{*} H=  {\phi^{\dagger}}_* (\phi^* (m D_k) -E ) = m D^{\dagger}_k$
since $E$ is contracted by $\phi^{\dagger} $ and $X  \dashrightarrow X^{\dagger}$ is an isomorphism in codimension $1$.
Thus $D^{\dagger}_k $ is $\Q$-cartier and semiample.

Let $D^{+}_i$ be the strict transform of $D_i$ on $X^{+}$ as in the statement of this lemma.
Since $\nu : X^+ \arw X^{\dagger}$ is an isomorphism in codimension $1$,
$D^+_k = \nu^* D^{\dagger}_k $ and $D^+_{k+1}= \nu^* D^{\dagger}_{k+1}$ hold.
Hence $D^+_k $ and $D^+_{k+1}$ are $\Q$-cartier and semiample.
Fix a prime divisor $D^+$ on $X^+$.
Since $D_k$ and $D_{k+1}$ span $N^1(X)_{\Q} = \Pic(X)_{\Q}$ as a $\Q$-vector space,
we can write $D \sim_{\Q} a D_k + b D_{k+1} $ for some $a,b \in \Q$,
where $D$ is the strict transform of $D^+$ on $X$.
Since $X  \dashrightarrow  X^+$ is an isomorphism in codimension $1$,
we have $D^+ \sim_{\Q} a D^+_k + b D^+_{k+1} $,
which means $D^+$ is $\Q$-Cartier.
Hence $X^+$ is $\Q$-factorial.

\ \\
\textbf{Step 3.}
First,
we show $D^+_1 \cap \cdots \cap D^+_k = \emptyset$.
Since $\nu(D^+_i)=D^{\dagger}_i$,
it is enough to show $D^{\dagger}_1 \cap \cdots \cap D^{\dagger}_k = \emptyset $.
By Step 1,
$X \setminus (D_1 \cap \cdots \cap D_k) = X \setminus C \cong X^{\dagger} \setminus \phi^{\dagger}(E)$.
Hence we have $(D^{\dagger}_1 \cap \cdots \cap D^{\dagger}_k) \setminus \phi^{\dagger}(E)=\emptyset $.
Therefore it suffices to see $(D^{\dagger}_1 \cap \cdots \cap D^{\dagger}_k) \cap \phi^{\dagger}(E)=\emptyset $.

For $1 \leq i \leq k$,
we have an effective Cartier divisor $H_i$ on $\P_{\pi(C)} \left( \bigoplus_{j=1}^k \calo_Z (m b_j \overline{D}_{k+1}) |_{\pi(C)} \right)$
by the natural projection
\[
\displaystyle \bigoplus_{j=1}^k  \calo_Z (m b_j \overline{D}_{k+1})
\arw \bigoplus_{j \not = i } \calo_Z (m b_j \overline{D}_{k+1}).
\]
By construction,
$D^{\dagger}_i \cap \phi^{\dagger}(E) =H_i \cap \phi^{\dagger}(E)$ holds as sets under the inclusion (\ref{shiki_inclusion2}).
Since $\bigcap_{i=1}^k H_i = \emptyset$,
we have
\[
\bigcap_{i=1}^{k} D^{\dagger}_i \cap \phi^{\dagger}(E)= \bigcap_{i=1}^{k} H_i \cap \phi^{\dagger}(E) = \emptyset.
\]
Hence $D^{\dagger}_1 \cap \cdots \cap D^{\dagger}_k$ is empty,
and so is $D^+_1 \cap \cdots \cap D^+_k$.

Since $D^+_{k}$ and $D^+_{k+1}$ are semiample by Step 2,
we can take an ample line bundle $A^+$ on $X^+$ such that $A^+ \in \Cone (D^+_k,D^+_{k+1})$.
Then we have $D^+_1 \succeq_{A^+} \cdots \succeq_{A^+} D^+_k $.
Since $D^+_1 \cap \cdots \cap D^+_k = \emptyset$,
we can apply Lemma \ref{exist_of edge} to $D^+_1, \ldots, D^+_k $
and obtain $ 0 \leq l \leq k-1$ such that
$D^+_1 , \ldots,  D^+_{l+1} $ satisfy condition $(*)$.
In particular,
$ D^+_{l+1}$ is semiample but not ample.
Since $D^+_{k+1}$ is semiample but not ample as well,
$\Nef(X^+) $ is spanned by two semiample divisors $D^+_{l+1} $ and $ D^+_{k+1}  $,
that is, $\Nef(X^+) = \r+ D^+_{l+1} + \r+ D^+_{k+1} $ holds.
\end{proof}

\begin{proof}[\bf{Proof of Theorem \ref{intro thm}}]

First,
we show that
if we relabel $D_1,\ldots,D_r$ and $D'_1,\ldots,D'_{r'} $ if necessary,
there exist integers $0 \leq k \leq r -1 $, $ 0 \leq k' \leq r'-1 $ such that
$D_1,\ldots,D_{k+1}$ and $D'_1,\ldots,D'_{k'+1}$ satisfy condition $(*)$ respectively.
In particular,
$D_{k+1}$ and $D'_{k'+1} $ are semiample and
$\Nef(X) = \r+ D_{k+1} + \r+ D'_{k'+1}  $.
Hence
$X$ satisfies ii) in Definition \ref{def_MDS}.

Let $A$ be an ample line bundle on $X$.
Assume $A \not \in \Cone (D_1,\ldots,D_{r})$.
By relabeling $D_1,\ldots,D_r$ if necessary,
we can apply Lemma \ref{exist_of edge} to $D_1,\ldots,D_r$ by the assumption b).
Hence there exists $ 0 \leq k \leq r-1$ such that $D_1,\ldots,D_{k+1}$ satisfy condition $(*)$.
Since $D_{k+1}$ is nef and not contained in $\Cone(D'_1,\ldots, D'_{r'})$,
there exists an ample line bundle $A' \not \in  \Cone(D'_1,\ldots, D'_{r'})$ by the assumption a).
Applying Lemma \ref{exist_of edge} to $D'_1,\ldots, D'_{r'}$,
we obtain $0 \leq k' \leq r'-1 $ such that $D'_1,\ldots,D'_{k'+1}$ satisfy condition $(*)$.
Since $ \r+ D_{k+1} \not = \r+ D'_{k'+1}$,
we have $\Nef(X) =  \r+ D_{k+1} + \r+ D'_{k'+1} $.

If $A \in \Cone (D_1,\ldots,D_{r})$,
then $A \not \in \Cone(D'_1,\ldots, D'_{r'})$ by the assumption a).
Hence we can apply the same argument.

\vspace{2mm}
To prove that $X$ is a Mori dream space,
it is enough to show that $X$ satisfies iii) in Definition \ref{def_MDS}.
In the rest of the proof,
we investigate the edges of $\Mov(X)$.

\ \\
\textbf{Case 1.}
First,
we consider the case $\dim D_1 \cap \cdots \cap D_{k} = \dim X$,
i.e.,
$k=0$.
In this case,
$D_1=D_{k+1}$ is semiample and not big by Lemma \ref{rem_not_ample}.
Hence $\r+ D_1$ is an edge of both $\Eff (X)$ and $\Mov(X)$.

\ \\
\textbf{Case 2.}
Next,
we consider the case $\dim D_1 \cap \cdots \cap D_{k} = \dim X -1$.
By Lemma \ref{rem_not_ample},
$\pi:=\varphi_{|m D_{k+1}|}$ is a birational morphism for sufficiently divisible $m$.
Fix a prime divisor $E$ on $X$ contained in $D_1 \cap \cdots \cap D_{k}$.
Then $E$ is contracted by $\pi$ since $\dim \pi (D_1 \cap \cdots \cap D_{k})  \leq \dim X -2$ by $3)$ in condition $(*)$ and $k \geq 1$.
We show that $\r+ D_{k+1}$ and $\r+ E$ are edges of $\Mov (X)$ and $\Eff (X)$ respectively.

Any prime divisor on $X$ contracted by $\pi$ is $\Q$-linearly equivalent to $a E$ for some $a > 0$
since $X$ is $\Q$-factorial, $N^1(X)_{\Q}=\Pic(X)_{\Q}$, and the Picard number is $2$.
On the other hand,
$E$ is not movable since $E$ is an exceptional divisor of $\pi$.
Hence $E$ is the unique exceptional divisor of $\pi$.

Let $Z$ be the image of $X$ by $\pi$.
For any prime divisor $F$ on $Z$,
the strict transform $\widetilde{F}$ of $F$ on $X$ is $\Q$-linearly equivalent to $\alpha D_{k+1} + \beta E$ for some $\alpha,\beta \in \Q$.
Since $\pi$ contracts $E$,
it holds that $F = \pi_* \widetilde{F} \sim_{\Q} \pi_* (\alpha D_{k+1} + \beta E) =\alpha \pi_*  D_{k+1} $.
Since ${\pi }_* D_{k+1}$ is $\Q$-Cartier,
$Z$ is $\Q$-factorial
and $\Pic (Z)_{\Q}$ is generated by ${\pi }_* D_{k+1}$.

For a prime divisor $D \not = E$ on $X$,
$\pi_*(D)$ is a prime divisor since $D$ is not contracted by $\pi$.
Thus we have $\pi^{*} \pi_{*}(D) \sim_{\Q} b \pi^{*} {\pi }_* D_{k+1} \sim b D_{k+1}$ for some $b > 0$.
Since $\pi^{*} \pi_*(D) -D $ is effective
and its support is contained in the exceptional locus of $\pi$,
we have $\pi^{*} \pi_*(D) -D \sim_{\Q} a E$ for some $a \geq  0$.
Thus $ D \sim_{\Q} -a E + b D_{k+1}$.
Hence $\r+ E$ is an edge of $\Eff (X)$ and
any movable divisor is contained in $- \r+ E + \r+ D_{k+1}$.
Thus $\r+ D_{k+1}$ is an edge of $\Mov(X)$.

\ \\
\textbf{Case 3.}
We consider the case $\dim D_1 \cap \cdots \cap D_{k} \leq \dim X -2$.
We set $X^{(0)}=X$.
In this case,  we can apply Proposition \ref{exist_SQM} to $D_1 , \ldots , D_{k+1}$
and obtain an SQM $X^{+}$
and $D^{+}_1,\ldots,D^{+}_{l+1}$ as in Proposition \ref{exist_SQM} for some $ 0 \leq l \leq k-1$.
Set $X^{(1)}=X^+$, $D^{(1)}_i = D^{+}_i$, and $k_1=l$.
If $\dim D^{(1)}_1 \cap \cdots \cap D^{(1)}_{k_1} \geq \dim X^{(1)} -1$,
we are reduced to Cases 1 or 2.
Hence $\r+ D_{k_1+1}$ is an edge of $\Mov(X^{(1)})=\Mov(X)$.

If $\dim D^{(1)}_1 \cap \cdots \cap D^{(1)}_{k_1} \leq \dim X^{(1)} -2$,
we can apply Proposition \ref{exist_SQM} to $D^{(1)}_1 , \ldots , D^{(1)}_{k_1+1}$
and obtain another SQM $X^{(2)}$ and $D^{(2)}_1,\ldots,D^{(2)}_{k_2+1}$ for some $ 0 \leq k_2 \leq k_1-1$.
Since $k > k_1 > k_2 > \cdots \geq 0$ is a decreasing sequence of non-negative integers,
we reach Cases 1 or 2 after repeating this process finitely many times (say, $m$ times)
and obtain SQMs $X^{(1)}, \ldots, X^{(m)}$
and an edge $\r+ D_{k_m +1}$ of $\Mov(X)$ for some $0 \leq k_m \leq k-1$.
By Proposition \ref{exist_SQM},
$\Nef(X^{(j)}) = \r+ D^{(j)}_{k_j+1} + \r+ D^{(j)}_{k_{j-1}+1}$
and $D^{(j)}_{k_j+1} , D^{(j)}_{k_{j-1}+1}$ are semiample for $1 \leq j \leq m$,
where we set $k_0=k$.

\vspace{2mm}
In Cases 1 and 2,
$\r+ D_{k+1}$ is an edge of $\Mov(X)$.
Following Case 3,
we set $m=0$, $k_0=k$, and $X^{(0)}=X$ in Cases 1 and 2.
Then $\r+ D_{k_m+1}$ is an edge of $\Mov(X)$ in Cases 1, 2, 3.

\vspace{3mm}
Applying the same argument to $D'_1,\cdots,D'_{k'+1}$,
we obtain SQMs $X'^{(0)}, \ldots, X'^{(m')}$ for some $m' \geq 0$
(note $ X= X^{(0)} = X'^{(0)}$ is the identity SQM)
and another edge $\r+ D'_{k'_{m'}+1}$ of $\Mov(X)$.
Then $\Mov(X)=\r+ D_{k_m+1} + \r+ D'_{k'_{m'}+1}$ is the union of $\Nef(X^{(j)}), \Nef(X'^{(j')})$ for $0 \leq j \leq m, 0 \leq j' \leq m'$.
Since the nefcone of each SQM is spanned by two semiample divisors,
$X$ satisfies iii) in Definition \ref{def_MDS}.
\end{proof}


\section{Examples}\label{section_example}


Corollaries \ref{cor 1}, \ref{blowup of Fano} are special cases of Corollary \ref{cor_c.i.}.
To show Corollary \ref{cor_c.i.},
we give divisors satisfying conditions a), b) in Theorem \ref{intro thm} explicitly. 
To clarify the idea,
we prove Corollary \ref{cor 1} first.

\begin{proof}[\bf{Proof of Corollary \ref{cor 1}}]
Since $\dim X = n \geq 3$,
it holds that $\Pic (X) \cong \Pic (\P^1 \times \P^n)$
by the Lefschetz hyperplane theorem.
We denote by $\calo_X(k,l)$ the restriction of $\calo_{\P^1 \times \P^n}(k,l)$ on $X$.

Let $u, v$ be homogeneous coordinates on $\P^1$.
Since $H^0(\P^1 \times \P^n, \calo(a,b) ) = H^0(\P^1,\calo(a))  \otimes H^0(\P^n,\calo(b))$,
$X$ is the zero section of
\[
u^a f_0 + u^{a-1} v f_1 + \cdots + v^a f_a
\]
for some general $f_i \in H^0(\P^n,\calo(b))$.
Set $W:=(f_0= \cdots = f_a=0) \subset \P^n$.
Since $f_i$ are general,
$\dim W =n -a-1 \geq -1$.

\vspace{1mm}
Let $\overline{D}_i \sim  \calo_X(i-1, b)  $ be the effective divisor on $X$ defined by
\[
(u^{i-1} f_0 +u^{i-2} v f_1 +  \cdots + v^{i-1} f_{i-1}) |_X 
\]
for $1 \leq i \leq a$.
Since
\[
u^{a-i+1} (u^{i-1} f_0 +u^{i-2} v f_1 +  \cdots + v^{i-1} f_{i-1}) |_X = - v^i (u^{a-i} f_i+ \cdots + v^{a-i} f_a) |_X,
\]
$D_i := \overline{D}_i  - (v^i=0)|_X \sim \calo_X(-1,b)$ is an effective divisor on $X$
for each $1 \leq i \leq a$.
By the definition of $D_i$,
we have
\[
D_1 \cap \cdots \cap D_a = (p_2^* f_0 = \cdots = p_2^* f_a =0) = p_2^{-1}(W) = \P^1 \times W,
\]
where $p_2 : \P^1 \times \P^n \arw \P^n$ is the second projection.
Choose general members $\overline{D}_{a+1}, \ldots, \overline{D}_{n} $ in $|\calo_{\P^n} (1)|$
and set $D_i = (p_2^* \overline{D}_i ) |_X$ for $a+1 \leq i \leq n$.
Since $W \cap \overline{D}_{a+1} \cap \cdots \cap \overline{D}_{n} = \emptyset$,
we have $D_1 \cap \cdots \cap D_n = \emptyset$.

If $D'_1, D'_2$ are general members in $| \calo_X(1,0) |$,
we have $D'_1 \cap D'_2= \emptyset$.
Since $D_1 , \ldots, D_n $ and $D'_1, D'_2$ satisfy a), b) in Theorem \ref{intro thm},
$X$ is a Mori dream space.
\end{proof}

Corollary \ref{cor 1} can be generalized by a similar argument as follows.

\begin{cor}\label{cor_c.i.}
Let $Y$ be a smooth Mori dream space with Picard number $1$,
and let $A$ be a base point free line bundle on $Y$.
Let $\pi : \P=\P_Y(\calo_Y \oplus A) \arw Y$ be the $\P^1$-bundle on $Y$ and  let $\calo_{\P}(1)$ be the tautological line bundle on $\P$.
Let $X = D^1 \cap \cdots \cap D^s \subset \P$ be a complete intersection on $\P$ of general divisors $D^j \in |\calo_{\P}(a_j) \otimes \pi^*B_j|$,
where $a_1, \ldots,a_s$ are positive integers and $B_1,\ldots,B_s$ are ample and base point free line bundles on $Y$.
If $\dim X \geq \max \{ \sum_{j=1}^s a_j ,3 \} $,
$X$ is a Mori dream space.
\end{cor}

\begin{proof}
Set $n=\dim X$.
Since $A$ is nef and $B_j$ is ample,
$D^j$ is ample.
Hence $ \Pic (X) \cong \Pic \P = \Z \calo_{\P}(1) \oplus \pi^{*} \Pic Y$ by the Lefschetz hyperplane theorem and $n \geq 3$.

Let $u \in H^0(\P, \calo_{\P}(1))$ and $v \in H^0(\P,\calo_{\P}(1) \otimes \pi^* A^{-1})$
be the sections corresponding to
the first and second summands of $\calo_Y \oplus A$ respectively.

Take $f^j \in H^0(\P, \calo_{\P}(a_j) \otimes \pi^*B_j)$ which defines $D^j$.
Since
\[
H^0(\P, \calo_{\P}(a_j) \otimes \pi^*B_j) \cong H^0(Y, \Sym^{a_j} (\calo \oplus A) \otimes B_j) \cong
\bigoplus_{i=0}^{a_j} H^0(Y,A^{\otimes i} \otimes B_j),
\]
we can write
\[
f^j = u^{a_j} f_0^j + u^{a_j-1} v f_1^j + \cdots + v^{a_j} f_{a_j}^j
\]
for some $ f_i^j \in H^0(Y, A^{\otimes i} \otimes B_j)$.
Since $X$ is general,
each $f_i^j$ is a general section of the base point free line bundle $ A^{\otimes i} \otimes B_j$.
For $1 \leq i \leq a_j$,
we define a divisor $D_i^j \sim \calo_{\P}(-1) \otimes \pi^* (A^{\otimes i} \otimes B_j) \, |_X $ on $X$ to be
\[
D_i^j = ( u^{i-1} f_0^j +u^{i-2} v f_1^j +  \cdots + v^{i-1} f_{i-1}^j = 0) |_X - (v^i =0) |_X ,
\]
which is effective as in the proof of Corollary \ref{cor 1}.
Similarly as in the proof of Corollary \ref{cor 1},
we have
\[
\bigcap_{1 \leq i \leq a_j, 1 \leq j \leq s} D_i^j = \pi^{-1} \left( \bigcap_{0 \leq i \leq a_j, 1 \leq j \leq s} (f_i^j=0) \right).
\]
Since $f_i^j$ are general and $\dim Y = n+ s -1$, it holds that
\[
\dim \bigcap_{i,j} (f_i^j=0) = \dim Y - \sum_j (a_j+1) = n -\sum a_j -1.
\]
Choose general hypersurfaces $\overline{D}_{1}, \ldots, \overline{D}_{ n -\sum a_j} $ on $Y$
and set $D_i = (\pi^* \overline{D}_i )|_X$ for $1 \leq i \leq  n -\sum a_j$.
Since $\overline{D}_{1} \cap \cdots \cap \overline{D}_{n -\sum a_j} \cap \bigcap_{i,j} (f_i^j=0)  = \emptyset$,
we have
\[
D_1 \cap \cdots \cap D_{n -\sum a_j} \cap  \bigcap_{1 \leq i \leq a_j, 1 \leq j \leq s} D_i^j = \emptyset.
\]

We set $D'_1=(u=0) |_X \sim \calo_{\P}(1) |_X $, $D'_2= (v=0) |_X \sim \calo_{\P}(1) \otimes \pi^* A^{-1} |_X$.
Then we have $D'_1 \cap D'_2= \emptyset$.
Since $D_1 , \ldots, D_{ n -\sum a_j} , \{ D_i^j \}_{i,j}$ and $D'_1, D'_2$ satisfy a), b) in Theorem \ref{intro thm},
$X$ is a Mori dream space.
\end{proof}

\begin{proof}[\bf{Proof of Corollary \ref{blowup of Fano}}]
When $\dim X=2$,
$X$ is a Del Pezzo surface.
Hence $X$ is a Mori dream space.

Thus we may assume $\dim X = \dim Z \geq 3$.
Set $n=N-s$.
For the blow up $\mu : \widetilde{\P}^{n+s} \arw \P^{n+s}$ at a general point $p \in Z$,
$X$ is a complete intersection on $\widetilde{\P}^{n+s} $
of general hypersurfaces in $|\mu^* \calo(d_1) - E| , \ldots, |\mu^* \calo(d_s) -E|$ since $Z$ and $p$ are general.
Let
\[
\pi  : \widetilde{\P}^{n+s} = \P_{\P^{n+s-1}} (\calo \oplus \calo(1)) \arw \P^{n+s-1}
\]
be the $\P^1$-bundle obtained from $|\mu^* \calo(1) - E |$.
Since $ \mu^* \calo(1)$ is the tautological bundle of $\pi$ and $ \mu^* \calo(1) -E = \pi^* \calo_{\P^{n+s-1}}(1)$,
we can apply Corollary \ref{cor_c.i.} to $Y=\P^{n+s-1}$, $A = \calo_{\P^{n+s-1} }(1)$,
$a_j = d_j-1$, and $B_j = \calo_{\P^{n+s-1} }(1)$ if $\dim X = n \geq \sum (d_j-1) $.
This condition is nothing but the ampleness of $-K_Z$ by the adjunction formula. 
\end{proof}

\begin{rem}
If $Z$ is not Fano,
Corollary \ref{blowup of Fano} does not hold in general.
For example,
$X$ is not a Mori dream space if $Z$ is a very general quartic surface in $\P^3$
and $p \in Z$ is a very general point by \cite[Proposition 6.3]{AL}.
We note that Proposition 6.3 in \cite{AL} claims that $X$ is not a Mori dream space for {\it some} $Z$ and $ p$,
but their proof works for very general $Z,p$.
\end{rem}

By checking the proof of Theorem \ref{intro thm} carefully,
we can explicitly describe cones in $N^1(X)_{\R}$ for Corollaries \ref{cor 1}, \ref{blowup of Fano}, or \ref{cor_c.i.}.
We illustrate the description by a special case of Corollary \ref{blowup of Fano}.
We leave the other cases to the reader.

\begin{eg}
Let $Z \subset \P^{n+1}$ be a general hypersurface of degree $n+1$ for $n \geq 3$.
Let $\mu : X \arw Z$ be the blow up at a general point $p \in Z$
and let $E$ be the exceptional divisor.
We set $H= \mu^* \calo_Z(1)$.
By the proofs of Corollaries \ref{blowup of Fano} and \ref{cor_c.i.},
we have effective divisors $D_i \sim i H -(i+1)E$ on $X$ for $1 \leq i \leq n$
such that $D_1 \cap \cdots \cap D_n = \emptyset$.
It is easy to see
$D_1,\ldots,D_n$ satisfy condition $(*)$ (see the proof of Lemma \ref{exist_of edge}),
hence $\Nef(X) = \r+ D_n + \r+ H$ as in the first paragraph of the proof of Theorem \ref{intro thm}.
In this case,
$k$ in the proof of Theorem \ref{intro thm} is $n-1$.

Since $ \dim D_1 \cap \cdots \cap D_{n-1} = 1 \leq \dim X -2$,
we obtain an SQM $X^{(1)}$ of $X^{(0)} :=X$ by Case 3 in the proof of Theorem \ref{intro thm}.
As in the proof of Theorem \ref{intro thm},
we denote by $D_i^{(j)}$ the strict transform of $D_i$ on $X^{(j)}$
(note that this $D_i^{(j)}$ is different from $D_i^j$ in the proof of Corollary \ref{cor_c.i.}).
By the proof of Proposition \ref{exist_SQM},
$k_1$ in Case 3 in the proof of Theorem \ref{intro thm} is $n-2$ in this case.
Hence $\Nef(X^{(1)}) = \r+ D_{n-1}^{(1)} + \r+ D_n^{(1)}$ holds.
By repeating this process,
we obtain an SQM $X^{(j)}$ for each $1 \leq j \leq n-2$
such that $D_1^{(j)}, \ldots, D_{n-j}^{(j)}$ satisfy condition $(*)$
and $\Nef(X^{(j)}) = \r+ D_{n-j}^{(j)} + \r+ D_{n+1-j}^{(j)}$.
On $X^{(n-2)}$,
we reach Case 2 in the proof of Theorem \ref{intro thm}.
Thus $D_1^{(n-2)}$ and $D_2^{(n-2)}$ are edges of $\Eff(X^{(n-2)})=\Eff(X) $ and $\Mov(X^{(n-2)})=\Mov(X)$ respectively.

From the above argument,
we have the following description: 

\begin{align*}
\Nef(X) = \Nef(X^{(0)}) &= \r+ D_n + \r+ H, \\
\Nef(X^{(j)}) &= \r+ D_{n-j} + \r+ D_{n+1-j} \quad \text{ for } \ 1 \leq j \leq n-2, \\
\Mov(X) &= \bigcup_{j=0}^{n-2} \Nef(X^{(j)}) = \r+ D_2 + \r+ H, \\
\Eff(X) &= \r+ D_1 + \r+ E.
\end{align*}

\end{eg}

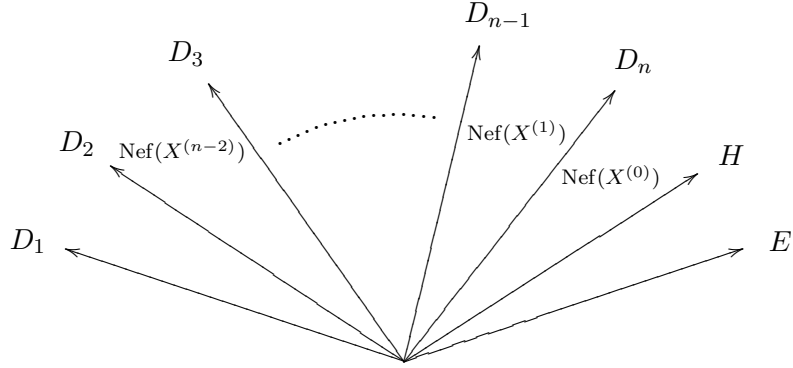
\begin{figure}[htbp]
 \begin{center}

\[
\begin{xy}
(0,0)="O",
(-45,15)="A",
(-39,26)="B",
(-26,37)="C",
(10,42)="D",
(28,36)="E",
(39,25)="F",
(45,15)="G",
(-29.5,28)*{{\scriptstyle \Nef(X^{(n-2)})}},
(27.5,24.5)*{ {\scriptstyle \Nef(X^{(0)})}},
(15,30.5)*{{\scriptstyle \Nef(X^{(1)})}},
(50,16)*{E},
(43.5,27.5)*{H},
(30.5,40)*{D_{n}},
(12.5,46)*{D_{n-1}},
(-29,41)*{D_3},
(-43.5,29)*{D_2},
(-50,16)*{D_1},
(-16.5,29);(4,32.5)

\ar "O";"A"
\ar "O";"B"
\ar "O";"C"
\ar "O";"D"
\ar "O";"E"
\ar "O";"F"
\ar "O";"G"

\crv{~*{.}(-7,34)}
\end{xy}
\]

 \end{center}
 \caption{Cones in $N^1(X)$}
 \label{figure2}
\end{figure}

\end{document}